\documentclass{amsart}
\usepackage{amsmath,amsthm,amssymb}
\usepackage{graphicx}

\newtheorem{theorem}{Theorem}[section]
\newtheorem{proposition}[theorem]{Proposition}
\newtheorem{lemma}[theorem]{Lemma}
\newtheorem{corollary}[theorem]{Corollary}

\theoremstyle{definition}

\newtheorem{remark}[theorem]{Remark}

\title[Infinitely many small exotic Stein fillings]{Infinitely many small exotic Stein fillings}
 
\author[Akbulut and Yasui]{Selman Akbulut and Kouichi Yasui}
\thanks{The first named author is partially supported by NSF grants DMS1065879 and DMS0905917. The second named author was partially supported by JSPS KAKENHI Grant Number 23840027.}
\date{September 12, 2012. Revised: April 8, 2013.}

\address{Department~of~Mathematics, Michigan State University, E.Lansing, MI, 48824, USA}
\email{akbulut@math.msu.edu}

\address{Department~of~Mathematics, Graduate School~of~Science, Hiroshima~University, 1-3-1 Kagamiyama, Higashi-Hiroshima, 739-8526, Japan}
\email{kyasui@hiroshima-u.ac.jp}

\begin{document}

\begin{abstract}We show that there exist infinitely many simply connected compact Stein 4-manifolds with $b_2=2$ such that they are all homeomorhic but mutually non-diffeomorphic,  and they are Stein fillings of the same contact 3-manifold on their boundaries. We also describe their handlebody pictures.
\end{abstract}

\maketitle

\section{Introduction}\label{sec:intro}
Stein 4-manifolds enjoy some rigidness properties which are useful in studying their topology. For example, the Stein manifolds, which certain closed 3-manifolds bound, are unique up to diffeomorphisms  $($cf.\ \cite{E2}, \cite{We}$)$. Hence identifying exotic (i.e.\ all homeomorphic but mutually non-diffeomorphic) Stein 4-manifolds is a particularly interesting problem. In ~\cite{AEMS} by using Lefschetz fibrations on knot surgered elliptic surfaces, Akhmedov-Etnyre-Mark-Smith constructed infinitely many exotic simply connected compact Stein fillings of Seifert fibered contact 3-manifolds. Recently Akhmedov-Ozbagci~\cite{AkhOz} generalized this example to a larger family of Seifert fibered contact 3-manifolds. These  4-manifolds have large second Betti numbers, so finding small concrete examples of such manifolds and describing their handlebody pictures is  a natural question in 4-manifold topology. In \cite{AY5}, generalization of the cork twisting constructons of \cite{AY2},  the authors constructed arbitrarily many exotic compact Stein $4$-manifolds which have the same topological invariants  of a given 2-handlebody $X$ with $b_2(X)\geq 1$,  though we do not know whether these exotic Stein manifolds induce the same contact $3$-manifolds on their boundaries.\smallskip


Here we construct infinitely many exotic Stein fillings of a fixed contact $3$-manifold with small second Betti number. We also give their Stein handlebody pictures. 
\begin{theorem}\label{sec:intro:thm:stein}
There exist infinitely many simply connected compact Stein 4-manifolds with $b_2=2$ such that they are all homeomorphic but mutually non-diffeomorphic, furthermore they are Stein fillings of the same contact 3-manifold. \end{theorem}
The construction is inspired by a recent paper \cite{A6} of the first author, where multiple log transforms produced infinitely many Stein fillings which are mutually exotic rel boundary. Here instead of multiple log transforms, we use single log transforms for the construction. This result should be contrasted with the result of \cite{Y5}, which says that any $p$-log transform $(p>1)$ along a homologically non-torsion $c$-embedded torus never produces a compact 4-manifold which admits a Stein structure.\vspace{0.5\baselineskip}\\
\textbf{Acknowledgements.} We would like to thank Chris Wendl for a helpful comment on contact structures.  This work was partially done during the workshop ``Invariants in Low-Dimensional Topology and Knot Theory'' at the Mathematisches Forschungsinstitut Oberwolfach in June 2012. We thank MFO staff and the organizers for the wonderful environment. We also thank the referee for his/her useful comments. 
\section{Construction}
Let $X$ be the compact smooth 4-manifold given by the left handlebody picture in Figure~\ref{fig1}. Note that $X$ contains an obvious $T^2 \times D^2$. Then by performing $X$ the $p$-log transform operation $(p\geq 1)$ in its interior 
we get $X_{p}$, which is the right picture. Here we used the description of $p$-log transform operation given in Section 4 of \cite{AY1}.  

\begin{figure}[ht!]
\begin{center}
\includegraphics[width=3.5in]{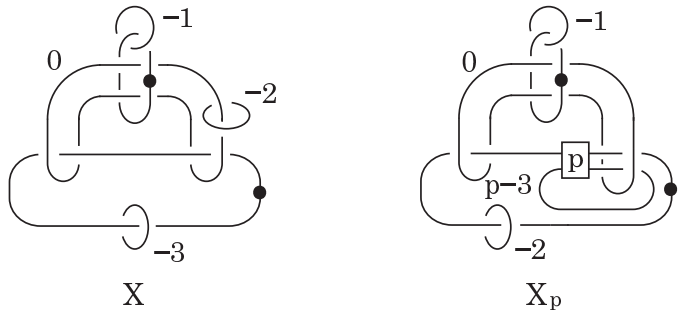}
\caption{$X$ and $X_{p}$}
\label{fig1}
\end{center}
\end{figure}
Note that $X$ and $X_{p}$ are simply connected and that their intersection forms are given by the matrices
\begin{equation*}
\left(
\begin{array}{cc}
0 & 1 \\
1 & -2
\end{array}
\right)\quad \textnormal{and}\quad 
\left(
\begin{array}{cc}
0 & 1 \\
1 & -2p^2+p-3
\end{array}
\right),
\end{equation*}
respectively. The intersection forms of $X$ and $X_{p}$ are thus unimodular and indefinite, which shows that $\partial X$ is a homology 3-sphere. It is easy to see that the form of $X_{p}$ is even if and only if $p$ is odd. Therefore, Freedman's theorem \cite{Fr} (cf.\ \cite{B}) together with the classification of the unimodular indefinite forms tells the following. 
\begin{lemma}\label{homeo} $(1)$ $X_{p}$ and $X_{q}$ are homeomorphic to each other if and only if the parities of $p$ and $q$ coincide.\smallskip\\
$(2)$ $X_{p}$ is homeomorphic to $X$ if and only if $p$ is a positive odd integer. 
\end{lemma}

Canceling the upper 1- and 2-handle pairs and converting the 1-handle notations, we get the Stein handlebody pictures of $X$ and $X_{p}$ in Figure~\ref{fig2} (For the time being, ignore symbols $\alpha, \beta, \gamma, \dots$.). It  follows from Eliashberg's theorem~\cite{E1}  that $X$ and $X_{p}$ $(p\geq 1)$ admit Stein structures. In the rest of the paper, we equip $X$ and $X_{p}$ with Stein structures given by these Legendrian pictures. 
\begin{figure}[ht!]
\begin{center}
\includegraphics[width=4.1in]{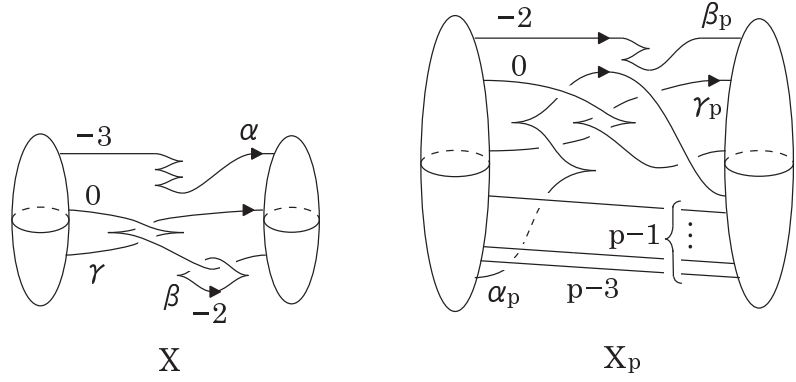}
\caption{Stein handlebodies of $X$ and $X_{p}$}
\label{fig2}
\end{center}
\end{figure}
\section{Detecting smooth structures}
Next we detect smooth structures by applying the adjunction inequality. The argument is a simplification of the genus arguments of \cite{AY4} and \cite{Y5}. 

Let $\alpha, \beta, \gamma$ and $\alpha_p, \beta_p, \gamma_p$ be the oriented attaching circles  (framed links) of the 2-handles of $X$ and $X_{p}$ as indicated in Figure \ref{fig2}. The rotation numbers of these circles are as follows. 
\begin{equation*}
r(\alpha)=2, \quad r(\beta)=0, \quad r(\gamma)=0,\quad r(\alpha_p)=-1, \quad r(\beta_p)=1, \quad r(\gamma_p)=0.
\end{equation*}

Define the basis $S,T$ of $H_2(X;\mathbb{Z})$ and the basis $R_p, T_p$ of $H_2(X_{p};\mathbb{Z})$ given by 2-handles as follows. 
\begin{equation*}
S=[\beta],\quad T=[\gamma],\quad R_p=[\alpha_p-p\beta_p], \quad T_p=[\gamma_p].
\end{equation*}
Furthermore define the class $S_p\in H_2(X_{p};\mathbb{Z})$ as follows. 
\begin{align*}
S_p&=\left\{
\begin{array}{ll}
R_{2q-1}+((2q-1)^2-q+1)T_{2q-1}, &\textnormal{if $p=2q-1$ for some integer $q$};\\
R_{2q}+((2q)^2-q+1)T_{2q}, &\textnormal{if $p=2q$ for some integer $q$}.
\end{array}
\right.
\end{align*}
Note that $T,S$ and $T_p, S_{p}$ are bases of $H_2(X;\mathbb{Z})$ and $H_2(X_{p};\mathbb{Z})$, respectively. The intersection matrices of $X$, $X_{2q-1}$ and $X_{2q}$ with respect to these bases are
\begin{equation*}
\left(
\begin{array}{cc}
0 & 1 \\
1 & -2
\end{array}
\right),
\quad 
\left(
\begin{array}{cc}
0 & 1 \\
1 & -2
\end{array}
\right)
\quad
\textnormal{and}\quad 
\left(
\begin{array}{cc}
0 & 1 \\
1 & -1
\end{array}
\right),
\end{equation*}
respectively. The following algebraic lemma is easily checked, using these bases. 
\begin{lemma}\label{basis restriction}Suppose $v\in H_2(X_{p};\mathbb{Z})$ satisfies $v\cdot v=\left\{
\begin{array}{ll}
-2, & \text{if $p$ is odd}, \\
-1, & \text{if $p$ is even}.
\end{array}
\right.$ 
Then $v=\pm S_p$. 
\end{lemma}

Let $g_{-1}$ (resp.\ $g_p$ $(p\geq 1)$) be the minimal number of genera of smoothly embedded closed surfaces which represent the class $S$ in $H_2(X;\mathbb{Z})$ (resp.\ $S_p$ in $H_2(X_{p};\mathbb{Z})$). 
\begin{lemma}\label{genus restriction}$(1)$ Let $p$ be a non-negative integer. Then $g_{2q-1}> g_{2p-1}$, for any integer $q>g_{2p-1}$. \smallskip\\
$(2)$ Let $p$ be a positive integer. Then $g_{2q}> g_{2p}$, for any integer $q>g_{2p}$. 
\end{lemma}
\begin{proof}$(1)$ Fix a non-negative integer $p$. Assume $q>g_{2p-1}$. 
As pointed out in \cite{AM}, the adjunction inequality holds for Stein 4-maifolds. The reason is as follows. By \cite{LM1}, we can holomorphically embed a Stein 4-manifold into a minimal closed complex surface of general type with $b_2^+>1$. Thus the adjunction inequality for closed 4-manifolds in \cite{KM1} and \cite{OzSz} gives the inequality for Stein manifolds. Therefore, we get the inequality 
 \begin{equation*}
2g_{2q-1}-2\geq |-1-(2q-1)|\, -2. 
\end{equation*}
The assumption of $q$ thus shows $g_{2q-1}> g_{2p-1}$. The case $(2)$ is similar. 
\end{proof}

\begin{corollary}\label{cor:exotic}Let $p$ be a positive integer. Then the following hold. \smallskip\\
$(1)$ $X_{2q-1}$ is homeomorphic but not diffeomorphic to $X_{2p-1}$ or $X$, for any $q> g_{2p-1}$.\smallskip\\
$(2)$ $X_{2q}$ is homeomorphic but not diffeomorphic to $X_{2p}$ for any $q> g_{2p}$. 
\end{corollary}
\begin{proof}$(1)$ Assume $q> g_{2p-1}$, and there exists a diffeomorphism $f: X_{2q-1}\to X_{2p-1}$. Lemma~\ref{basis restriction} shows $f_*(S_{2q-1})=\pm S_{2p-1}$. We thus have $g_{2q-1}= g_{2p-1}$, which contradicts Lemma~\ref{genus restriction}. The claim thus follows from Lemma~\ref{homeo}. The same argument holds for the $X$ case. The (2) case is similar. 
\end{proof}
To prove the main theorem, we use the following fact, which was kindly pointed to us by Wendl. 
\begin{lemma}[Wendl~\cite{We2}]\label{lem:Wendl}For any closed connected oriented 3-manifold, it has at most finitely many different strongly fillable contact strctures up to isomorphisms. 
\end{lemma}
\begin{proof}Theorem 0.6 in \cite{CGH} tells that for any given non-negative integer,
every closed connected 3-manifold has at most finitely many contactomorphism classes of tight
contact structures with Giroux torsion equal to that integer. The claim thus follows from Corollary 3 in \cite{Gay} which says that a contact structure with Giroux torsion $> 0$ is not strongly fillable. 
\end{proof}
Since a Stein filling is a strong filling, this lemma and the above corollary imply the following, which shows the main theorem.  
\begin{corollary}\label{cor:exotic fillings}$(1)$ There exists a contact structure $\xi$ on the boundary $\partial X$ such that at least infinitely many of $X_{2p-1}$'s $(p\geq 1)$ are Stein fillings of $(\partial X,\xi)$ and that these fillings are all homeomorphic but mutually non-diffeomorphic.\smallskip\\
$(2)$ There exists a contact structure $\eta$ on the boundary $\partial X$ such that at least infinitely many of $X_{2p}$'s $(p\geq 1)$ are Stein fillings of $(\partial X,\eta)$ and that these fillings are all homeomorphic but mutually non-diffeomorphic.
\end{corollary}
\begin{remark}Corollary~\ref{cor:exotic fillings}.(1) and (2) give spin and non-spin examples, respectively.
\end{remark}
\section{Alternative proof}
In this section, we give an alternative proof of Theorem~\ref{sec:intro:thm:stein}. Unlike the above proof, we directly prove that the boundary contact structures of exotic Stein 4-manifolds are isomorphic to each other without applying Lemma~\ref{lem:Wendl}, though we use Corollary~\ref{cor:exotic} to detect smooth structures. Beware that we use the symbols $\alpha, \beta, \gamma, \xi, \eta$ in the previous sections for different meanings.\medskip

Consider the 3-torus $T^3=T^2\times \partial D^2$. We regard $H_1(T^3;\mathbb{Z})$ as $H_1(S^1;\mathbb{Z})\oplus H_1(S^1;\mathbb{Z}) \oplus H_1(\partial D^2;\mathbb{Z})\cong \mathbb{Z}\oplus \mathbb{Z}\oplus \mathbb{Z}$. Recall that the isotopy classes of self-diffeomorphisms of $T^3$ are determined by their induced automorphisms of $H_1(T^3;\mathbb{Z})$. For a positive integer $p$, let $f_p:T^2 \times \partial D^2\to T^2 \times \partial D^2$ be a diffeomorphism which induces the automorphism of $H_1(T^3;\mathbb{Z})$ represented by the matrix $\left(
\begin{array}{ccc}
1 & 0 & 0 \\
0 & 1 & 0 \\
0 & p & 1
\end{array}
\right)
$. Two handlebody descriptions of $f_p$ are described in Figures~\ref{fig3} and \ref{fig3.5}, where $\alpha, \beta, \gamma$ denote circles in $T^3$.  These pictures were essentially given in \cite{AY1} and \cite{A5}, respectively. 

\begin{figure}[ht!]
\begin{center}
\includegraphics[width=4.3in]{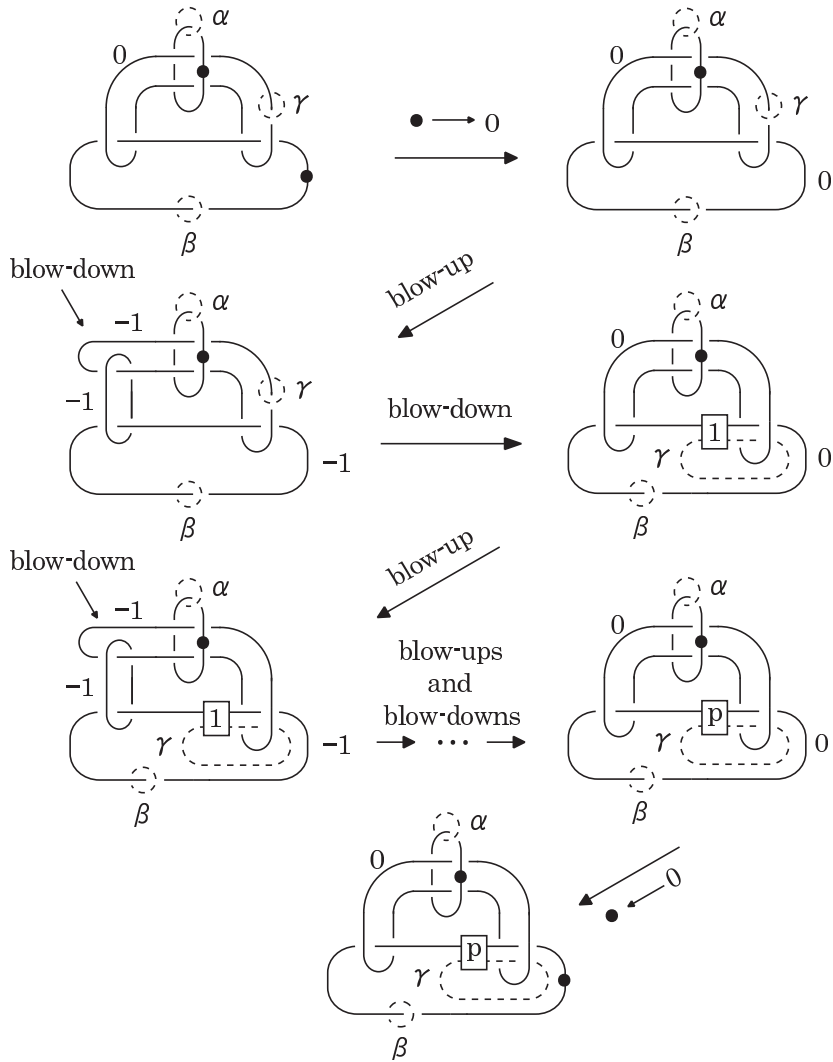}
\caption{A self-diffeomorphism $f_p$ of $T^2\times \partial D^2$}
\label{fig3}
\end{center}
\end{figure}

\begin{figure}[ht!]
\begin{center}
\includegraphics[width=4.5in]{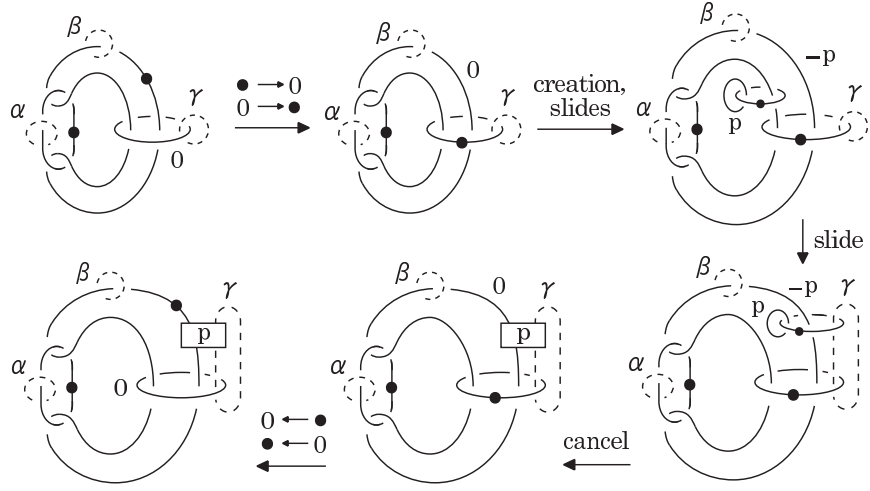}
\caption{A self-diffeomorphism $f_p$ of $T^2\times \partial D^2$}
\label{fig3.5}
\end{center}
\end{figure}

Now let $Y$ be the smooth 4-manifold given by the left handlebody in Figure~\ref{fig4}. Note that $Y$ contains an obvious $T^2\times D^2$. By removing the obvious $T^2\times D^2$ from $Y$ and regluing it via the diffeomorphism $f_p$, we get a 4-manifold $Y_p$, which is the right handlebody in the figure. Note that $Y_p$ is diffeomorphic to the $p$-log transform $X_p$ of $X$, since their handlebody diagrams are the same.\medskip

\begin{figure}[ht!]
\begin{center}
\includegraphics[width=3.5in]{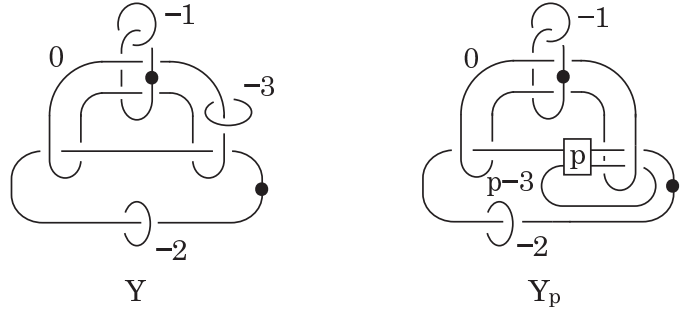}
\caption{$Y$ and $Y_{p}$}
\label{fig4}
\end{center}
\end{figure}

In the rest, we equip $Y$ and $Y_p$'s with Stein structures. Converting the 1-handle notation of $Y$, we get Figure~\ref{fig5}. By twisting the $0$-framed knot around the lower left 3-ball and putting the framed link into a Legendrian position, we get the handlebody of $Y$ in Figure~\ref{fig6}. Notice that the two 1-handles and the 0-framed Legendrian knot constitute a Stein handlebody decomposition of $T^2\times D^2$. Here recall that $T^2\times \partial D^2$ admits a unique Stein fillable contact structure $\xi$ up to  isomorphism (\cite{E3}). The rest of the three framed knots in the figure clearly constitute a Legendrian link (call it $L$) in $(T^3,\xi)$ such that each framing is one less than the Thurston-Bennequin number. Eliashberg's theorem~\cite{E1} thus tells that $Y$ admits a Stein structure. 

\begin{figure}[ht!]
\begin{center}
\includegraphics[width=1.7in]{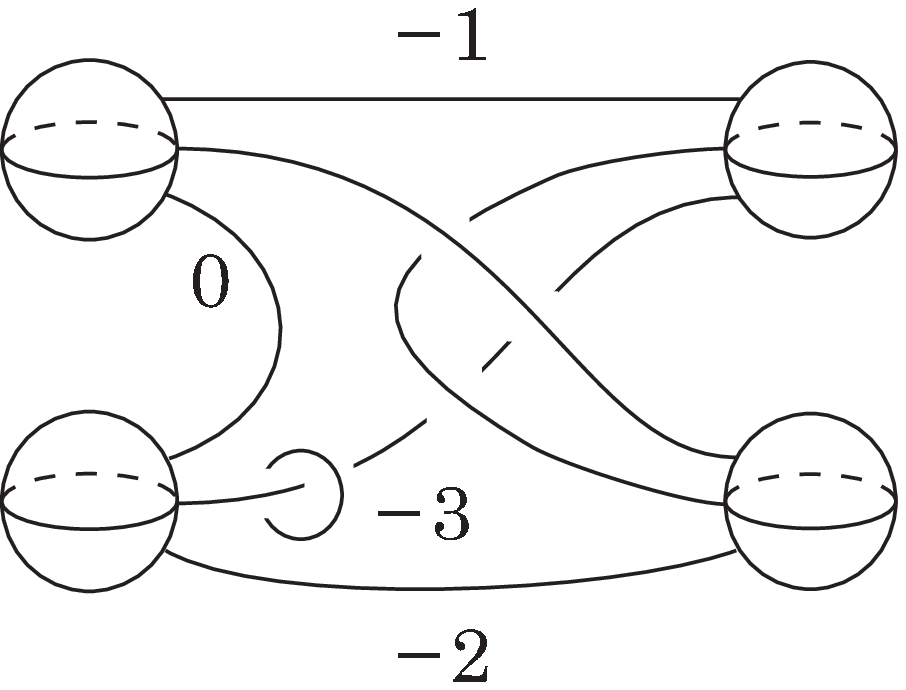}
\caption{$Y$}
\label{fig5}
\end{center}
\end{figure}

\begin{figure}[ht!]
\begin{center}
\includegraphics[width=2.1in]{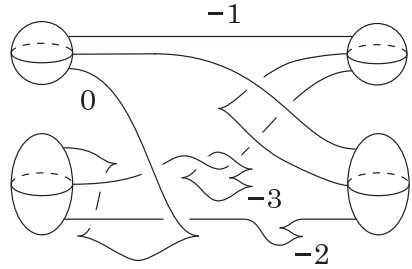}
\caption{A Stein handlebody decomposition of $Y$}
\label{fig6}
\end{center}
\end{figure}

Eliashberg-Polterovich~\cite{ElP} proved that if a diffeomorphism $g: T^3\to T^3$ stabilizes $\mathbb{Z}\oplus \mathbb{Z}\oplus 0\subset H_1(T^3;\mathbb{Z})$, then $g$ is isotopic to a self contactomorphism of $(T^3,\xi)$. We may thus regard the above $f_p$ as a self contactomorphism of $(T^3,\xi)$. Hence, $Y_p$ is obtained from $T^2\times D^2$ by attaching 2-handles along the Legendrian link $f_p(L)$ in $(T^3,\xi)$ with the contact $-1$ framings. Thus $Y_p$ admits a Stein structure. \medskip

The induced contact structure $\eta$ (resp.\ $\eta_p$) on $Y$ (resp.\ $\partial Y_p$) is obtained from $(T^3,\xi)$ by the contact $-1$ Dehn surgery along $L$ (resp.\ $f_p(L)$) (\cite{DG1}, cf.\ \cite{DG2}). Since $f_p$ is a contactomorphism, each $(\partial Y_p, \eta_p)$ is contactomorphic to $(\partial Y, \eta)$. Therefore, this fact together with Corollary~\ref{cor:exotic} gives the following, which shows Theorem~\ref{sec:intro:thm:stein}.
\begin{proposition}Every $4$-manifold $Y_p$ $(p\geq 1)$ is a simply connected Stein filling with $b_2=2$ of the same contact $3$-manifold $(\partial Y, \eta)$, which is a homology 3-sphere. Furthermore, the following hold.\medskip\\
$(1)$ $Y_{2p}$'s $(p\geq 1)$ are mutually homeomorphic and have odd intersection forms. Moreover, at least infinitely many of $Y_{2p}$'s are mutually non-diffeomorphic.\medskip\\
$(2)$ $Y_{2p-1}$'s $(p\geq 1)$ are mutually homeomorphic and have even intersection forms. Moreover, at least infinitely many of $Y_{2p-1}$'s are mutually non-diffeomorphic.
\end{proposition}
 
\begin{remark}$(1)$ The constructions in this paper clearly have many variations. For example, we can change link types and framings of the attaching circles of 2-handles of $X$ attached to $T^2\times D^2$. A more general construction in $b_2\geq 2$ case and the non-simply connected case will be discussed in \cite{Y6} applying techniques in \cite{Y5}.\medskip\\
$(2)$ Modifying the above construction, we can easily construct many contact 3-manifolds such that each of them has infinitely many mutually non-homeomorphic Stein fillings. The first such examples were given by Ozbagci-Stipsicz~\cite{OS2} and Smith~\cite{Sm}. 

Our construction is as follows. Consider a Legendrian knot $\widetilde{\gamma}$ in $(T^2 \times \partial D^2, \xi)$ that is smoothly isotopic to $\gamma$ in the first picture of Figure~\ref{fig3}. Let $V$ be the Stein 4-manifold obtained by attaching a 2-handle to $T^2 \times D^2$ along $\widetilde{\gamma}$ with the contact $-1$ framing. Remove the obvious $T^2\times D^2$ from $V$ and reglue it via the diffeomorphism $f_p$. Then, the resulting Stein manifold $V_p$ and $V$ induce the same contact structure up to isomorphism. On the other hand, we see $H_1(V_p;\mathbb{Z})\cong \mathbb{Z}\oplus (\mathbb{Z}/p\mathbb{Z})$. Therefore, $V_p$'s $(p\geq 1)$ are infinitely many mutually non-homeomorphic Stein fillings of the same contact 3-manifold. Note that $V$ is a plumbing of the disk bundle over a torus with Euler number zero and the disk bundle over a 2-sphere with Euler number $\leq -2$. There are clearly many variations. 
\medskip\\
$(3)$ We can construct arbitrarily many mutually non-isomorphic Stein fillable contact structures on a homology 3-sphere such that each of them has infinitely many simply connected exotic Stein fillings with $b_2=2$. The construction is roughly as follows. Change framings of the attaching circles $L$ of 2-handles of $Y$ attached to $T^2\times D^2$ into large negative numbers. Since there are many Legendrian realizations of the link $L$ in $(T^3,\xi)$, we get finitely many non-isomorphic Stein fillable contact structures on $\partial Y$, which can be detected by their $d_3$ invariants. Similarly to the proof in this section, we obtain infinitely many exotic Stein fillings for each of these contact structures. 
\end{remark}



\end{document}